\documentclass[11pt]{article}
\usepackage{amsmath, amssymb, amsfonts, amsthm, latexsym}
\usepackage{graphicx}
\usepackage{authblk}
\usepackage{hyperref}
\usepackage[usenames]{color}
\newtheorem{theorem}{Theorem}[section]
\newtheorem{lemma}[theorem]{Lemma}
\newtheorem{claim}[theorem]{Claim}
\newtheorem{proposition}[theorem]{Proposition}
\newtheorem{corollary}[theorem]{Corollary}
\newtheorem{remark}[theorem]{Remark}

\title{Semi-random process without replacement}
\author{Shoni Gilboa\thanks{Department of Mathematics and Computer Science, The Open University of Israel, Raanana 43107, Israel.} \and Dan Hefetz \thanks{Department of Computer Science, Ariel University, Ariel 40700, Israel. Research supported by ISF grant 822/18.}}

\begin{document}
\maketitle
\begin{abstract}
Semi-random processes involve an adaptive decision-maker, whose goal is to achieve some pre-determined objective in an online randomized environment. We introduce and study a semi-random multigraph process, which forms a no-replacement variant of the process that was introduced in~\cite{BHKPSS}.
The process starts with an empty graph on the vertex set $[n]$. 
For every positive integers $q$ and $1\leq r\leq n$, in the $((q-1)n+r)$th round of the process, the decision-maker, called \emph{Builder}, is offered the vertex $\pi_q(r)$, 
where $\pi_1, \pi_2, \ldots$ is a sequence of permutations in $S_n$, chosen independently and uniformly at random.
Builder then chooses an additional vertex (according to a strategy of his choice) and connects it by an edge to $\pi_q(r)$. 

For several natural graph properties, such as $k$-connectivity, minimum degree at least $k$, and building a given spanning graph (labeled or unlabeled), we determine the typical number of rounds Builder needs in order to construct a graph having the desired property. Along the way we introduce and analyze an urn model which may also have independent interest.
\end{abstract}

\section{Introduction}

In this paper we introduce and analyze a general semi-random multigraph process, arising from an interplay between a sequence of random choices on the one hand, and a strategy of our choice on the other. It is a no-replacement variant of the process which was proposed by Peleg Michaeli, analyzed in~\cite{BHKPSS}, and further studied in~\cite{BGHK, GKMP, BMPR, GMP}.
Denote by $S_n$ the set of permutations of the set $[n] := \{1,2, \ldots, n\}$,
and let $\pi_1, \pi_2, \ldots$ be a sequence of permutations in $S_n$, chosen independently and uniformly at random.
The process starts with an empty graph on the vertex set $[n]$. 
For every positive integer $k$, in the $k$th round of the process, the decision-maker, called \emph{Builder}, is offered the vertex $v_k := \pi_q(r)$, where $q$ and $r$ are the unique integers satisfying $(q-1)n+r=k$ and $1\leq r\leq n$.
Builder then irrevocably chooses an additional vertex $u_k$ and adds the edge $u_k v_k$ to his (multi)graph, with the possibility of creating multiple edges (in fact, we will make an effort to avoid multiple edges; allowing them is a technical aid which ensures that Builder always has a legal edge to claim). 

The algorithm that Builder uses in order to add edges throughout this process is referred to as Builder's \emph{strategy}.

Given a positive integer $n$ and a family $\mathcal F$ of labeled graphs on the vertex set $[n]$, we consider the one-player game in which Builder's goal is to build a multigraph with vertex set $[n]$ that contains, as a (spanning) subgraph, some graph from $\mathcal F$, as quickly as possible; we denote this game by $({\mathcal F}, n)_{\text{lab}}$. 
In the case that the family $\mathcal F$ consists of a single graph $G$, we will use the abbreviation $(G,n)_{\text{lab}}$  for $(\{G\},n)_{\text{lab}}$.
We also consider the one-player game in which Builder's goal is to build a multigraph with vertex set $[n]$ that contains a subgraph which is \emph{isomorphic} to some graph from $\mathcal F$, as quickly as possible; we denote this game by $({\mathcal F}, n)$. Note that 
\begin{equation}\label{eq:lab_vs_iso_game}
({\mathcal F}, n)=({\mathcal F}_{\text{iso}}, n)_{\text{lab}},
\end{equation}
where ${\mathcal F}_{\text{iso}}$ is the family of all labeled graphs on the vertex set $[n]$ which are isomorpic to some graph from $\mathcal F$. 
The general problem discussed in this paper is that of determining the typical number of rounds Builder needs in order to construct such a multigraph under optimal play. 

\paragraph*{A formal treatment.}
Suppose that Builder follows some fixed strategy $\mathcal S$. Let $\mathcal S(n,m)$ denote the resulting multigraph if Builder follows $\mathcal S$ for $m$ rounds. That is, $\mathcal S(n,m)$ is the probability space of all multigraphs with vertex set $[n]$ and with $m$ edges, where each of these edges is chosen as follows. 
For every positive integer $k$, in the $k$th round of the process, Builder is offered the vertex $v_k := \pi_q(r)$, where $q$ and $r$ are the unique integers satisfying $(q-1)n+r=k$ and $1\leq r\leq n$.
Builder then chooses the vertex $u_k$ according to $\mathcal S$, and the edge $u_k v_k$ is added to his graph. 

For the labeled game $({\mathcal F}, n)_{\text{lab}}$ and a strategy $\mathcal{S}$, let $\tau(\mathcal{S})$ denote the total number of rounds played until Builder's graph first contains some graph from $\mathcal F$, assuming he plays according to $\mathcal{S}$. In other words, $\tau(\mathcal{S})$ is the smallest integer $m$ for which the graph ${\mathcal S}(n,m)$ contains some graph from $\mathcal F$. For completeness, if no such integer $m$ exists, we define $\tau(\mathcal{S})$ to be $+\infty$. Note that $\tau(\mathcal S)$ is  a random variable. 
Let $p_{\mathcal S}$ be the non-decreasing function from the set $\mathbb N$ of non-negative integers to the interval $[0,1]$ defined by $p_{\mathcal S}(k) = {\Pr}(\tau({\mathcal S})\leq k)$ for every non-negative integer $k$.
Following~\cite{BHKPSS}, we say that $\mathcal{S}$ \emph{dominates} another strategy $\mathcal{S}'$ if $p_{\mathcal S}(k) \geq p_{\mathcal{S}'}(k)$ for every $k$. A strategy $\mathcal{S}$ is said to be \emph{optimal}, if it dominates any other strategy $\mathcal{S}'$. 
For every non-negative integer $k$, let $p_{({\mathcal F},n)_{\text{lab}}}(k)$ be the maximum of $p_{\mathcal S}(k)$, taken over all possible strategies $\mathcal{S}$ for $({\mathcal F}, n)_{\text{lab}}$.
Clearly, $p_{({\mathcal F},n)_{\text{lab}}}$ is a non-decreasing function from $\mathbb N$ to $[0,1]$; hence there exists a random variable $\tau_{\text{lab}}({\mathcal F},n)$ taking values in ${\mathbb N} \cup \{+\infty\}$ such that ${\Pr}(\tau_{\text{lab}}({\mathcal F},n)\leq k) = p_{({\mathcal F},n)_{\text{lab}}}(k)$ for every non-negative integer $k$.
Note that if there is an optimal strategy $\mathcal{S}$ for the labeled game $({\mathcal F}, n)_{\text{lab}}$, then we may take $\tau_{\text{lab}}({\mathcal F},n)$  to be $\tau(\mathcal S)$.

For the unlabeled game $({\mathcal F}, n)$ we define $\tau({\mathcal F},n)$ in an analogous manner, or by using \eqref{eq:lab_vs_iso_game}, namely 
\begin{equation}\label{eq:lab_vs_iso_tau_def}
\tau({\mathcal F},n) = \tau_{\text{lab}}({\mathcal F}_{\text{iso}},n).
\end{equation}
Since, obviously, $p_{({\mathcal F},n)}(k)\geq p_{({\mathcal F},n)_{\text{lab}}}(k)$ for every $k$, we may assume (by coupling) that 
\begin{equation}\label{eq:lab_vs_iso_tau_compare}
\tau({\mathcal F}, n) \leq \tau_{\text{lab}}({\mathcal F}, n).
\end{equation}

For a given family ${\mathcal F}$ of graphs on the vertex set $[n]$, our prime objective for the game $({\mathcal F}, n)_{\text{lab}}$ is to obtain tight upper and lower bounds on $\tau_{\text{lab}}({\mathcal F}, n)$ which hold with high probability (w.h.p.~for brevity), i.e., with probability which tends to $1$ as $n$ tends to $\infty$. 
Note that in order to prove that w.h.p. $\tau_{\text{lab}}({\mathcal F},n) \leq m$, it suffices to present a strategy $\mathcal S$ such that w.h.p. $\mathcal S(n,m)$ cotains a graph of ${\mathcal F}$. On the other hand, in order to prove that w.h.p. $\tau_{\text{lab}}({\mathcal F},n) > m$, one has to show that for any strategy $\mathcal S$, w.h.p. the graph $\mathcal S(n,m)$ does not contain any graph of ${\mathcal F}$.
Our prime objective for the unlabelled version of the game is analogous, namely, to obtain tight upper and lower bounds on $\tau({\mathcal F}, n)$ which hold with high probability.

In this paper we will establish such lower and upper bounds on $\tau_{\text{lab}}({\mathcal F}, n)$ and on $\tau({\mathcal F}, n)$ for several natural families ${\mathcal F}$ of graphs. 

The rest of the paper is organized as follows. Section~\ref{sec::prelim} introduces some notation and a technical result which will be used later on. Section~\ref{sec::general} contains two very simple but very general results; these will determine our focus for the rest of the paper. In Section~\ref{sec::urn} we introduce and analyze an urn model; it will be used in later sections, but may also have independent interest. In section~\ref{sec::MinDeg} we study minimum degree games. In Section~\ref{sec::PM} we study the construction of labeled and unlabeled perfect matchings. In Section~\ref{sec::oddRegularGraphs} we study the construction of labeled regular graphs. In Section~\ref{sec::trees} we study the construction of labeled and unlabeled trees. In Section~\ref{sec::EdgeCon} we study edge-connectivity games. Finally, in Section~\ref{sec::openProbs}, we suggest several open problems for future study.

\section{Notation and Preliminaries} \label{sec::prelim}

For every positive integer $m$, let
$$
H_m = \sum_{k=1}^m \frac{1}{k}
$$
denote the $m$th harmonic number. For all positive integers $\ell \leq m$, let
$$
H_{\ell,m} = H_m - H_{\ell} = \sum_{k=\ell+1}^m \frac{1}{k}.
$$
For all positive integers $\ell \leq m$ it clearly holds that
\begin{equation}\label{eq:harmonic1}
\frac{m-\ell}{m} \leq H_{\ell,m} \leq \frac{m-\ell}{\ell+1}.
\end{equation}
Moreover, since $\frac{1}{x+1} < \ln(x+1) - \ln x < \frac{1}{x}$ for every $x > 0$, for all positive integers $ \ell \leq m$ it holds that
\begin{equation}\label{eq:harmonic2}
\ln \left(\frac{m+1}{\ell+1} \right) = \ln(m+1) - \ln(\ell+1) \leq H_{\ell,m} \leq \ln m - \ln\ell = \ln \left(\frac{m}{\ell} \right).
\end{equation}

The following simple technical claim will be useful in Sections~\ref{sec::PM} and~\ref{sec::oddRegularGraphs}. 

\begin{claim} \label{claim:match_comp}
Let $n$ be an even positive integer. For $0 \leq r \leq n/2$ let 
$$
p_r := \Pr(\{\pi(1), \pi(2), \ldots, \pi(n-r)\} \cap \{2i-1,2i\} \neq \emptyset \text{ for every } 1 \leq i \leq n/2\}
$$
where $\pi \in S_n$ is chosen uniformly at random. Then
$$
p_r = \begin{cases}
1 - o(1) & r = o(\sqrt{n})\\
o(1) & r = \omega(\sqrt{n}).
\end{cases}
$$
\end{claim}

\begin{proof}
Note that for a permutation $\pi\in S_n$ it holds that $\{\pi(1), \pi(2), \ldots, \pi(n-r)\} \cap \{2i-1,2i\} \neq \emptyset$ for every $1\leq i\leq n/2$ if and only if $|\{\pi(n-r+1), \ldots, \pi(n)\} \cap \{2i-1,2i\}| \leq 1$ for every $1\leq i\leq n/2$. In order to count the permutations which satisfy the latter, we choose $r$ of the pairs $\{\{2i-1,2i\} : 1\leq i\leq n/2\}$ and then one element from each chosen pair; the $r$ chosen elements form the image of $\{n - r + 1, \ldots, n\}$ and the remaining $n-r$ elements form the image of $\{1, \ldots, n-r\}$. It follows that
\begin{align*}
p_r &= \frac{\binom{n/2}{r} 2^r (n-r)! r!}{n!} 
= \frac{\frac{n}{2} (\frac{n}{2} - 1) \cdots (\frac{n}{2} - r + 1) 2^r}{n (n-1) \cdots (n-r+1)} \\
&= \frac{n (n - 2) \cdots (n - 2r + 2)}{n (n-1) \cdots (n-r+1)} = \prod_{i=0}^{r-1} \left(1 - \frac{i}{n-i} \right).
\end{align*} 
Hence, if $r = \omega(\sqrt{n})$, then
$$
p_r = \prod_{i=0}^{r-1} \left(1 - \frac{i}{n-i} \right) \leq \exp \left\{- \sum_{i=0}^{r-1} \frac{i}{n-i} \right\}\leq \exp\left\{-\frac{(r-1)r/2}{n}\right\} = o(1),
$$
and if $r = o(\sqrt{n})$, then
\begin{align*}
p_r &= \prod_{i=0}^{r-1} \left(1 - \frac{i}{n-i} \right)=\prod_{i=0}^{r-1}\frac{1}{1 + \frac{i}{n-2i} }\geq \exp\left\{- \sum_{i=0}^{r-1} \frac{i}{n-2i}\right\}\\
&\geq \exp \left\{- \frac{(r-1)r/2}{n-2(r-1)} \right\}= 1 - o(1).
\qedhere\end{align*} 
\end{proof}

\section{General bounds} \label{sec::general}

The following two results are very simple but very widely applicable. Together with their many corollaries they form a good indication of what is interesting to prove in relation to the no-replacement semi-random process. 

\begin{proposition} \label{prop::orientationUB}
Let $G$ be a graph on the vertex set $[n]$. If there exists an orientation $D$ of the edges of $G$ such that $d_D^+(u) \leq d$ for every $u \in [n]$, then $\tau(G, n) \leq \tau_{\text{lab}}(G, n) \leq d n$.
\end{proposition}

\begin{proof}
Let $D$ be an orientation of the edges of $G$ such that $d_D^+(u) \leq d$ for every $u \in [n]$. For every $u \in [n]$, let $\tilde{u}^1, \ldots, \tilde{u}^{d_D^+(u)}$ be an arbitrary ordering of the vertices of $N_D^+(u)$. For every $1 \leq i \leq d n$, let $u_i$ denote the vertex Builder is offered in the $i$th round. In the $i$th round, Builder claims the edge $u_i \tilde{u}_i^k$, where $k \leq d_D^+(u_i)$ is the smallest integer for which $u_i \tilde{u}_i^k$ is free; if no such $k$ exists, then Builder claims an arbitrary edge which is incident with $u_i$. Since $d_D^+(u) \leq d$ for every $u \in [n]$ and since, during the first $d n$ rounds, every vertex of $[n]$ is offered precisely $d$ times, it readily follows from the description of Builder's strategy that, after $d n$ rounds, Builder's graph contains $G$ as a subgraph.       
\end{proof}

\begin{proposition} \label{prop::orientationLB}
Let $G$ be a graph on the vertex set $[n]$. Let $d$ be the largest integer such that in every orientation of the edges of $G$ there exists a vertex of out-degree at least $d$. Then
$$
\tau_{\text{lab}}(G, n)\geq \tau(G, n) \geq \max \{(d-1) n + 1, e(G)\}.
$$
\end{proposition}

\begin{proof}
Trivially, $\tau(G, n) \geq e(G)$; we will prove that $\tau(G, n) \geq (d-1) n + 1$ as well. Suppose for a contradiction that there exist permutations $\pi_1, \ldots, \pi_{d-1} \in S_n$ and a strategy $\mathcal{S}$ such that, if the vertices are offered according to $\pi_1, \ldots, \pi_{d-1}$ and Builder follows $\mathcal{S}$, then he builds a copy of $G$ within $(d-1) n$ rounds. Orient each edge Builder claims from the vertex he was offered to the vertex he chose to connect it to. Observe that the maximum out-degree in Builder's graph after $(d-1) n$ rounds is $d-1$ and thus, by the definition of $d$, his graph cannot admit a copy of $G$, contrary to our assumption.
\end{proof}

\begin{remark}
It is well-known (see Lemma 3.1 in~\cite{AT}) that a graph $G$ admits an orientation in which the out-degree of every vertex is at most $d$ if and only if $d \geq L(G)$, where 
$$
L(G) := \max \left\{\frac{e(H)}{v(H)} : \emptyset \neq H \subseteq G \right\}.
$$
\end{remark}

As noted above, despite being very simple, Propositions~\ref{prop::orientationUB} and~\ref{prop::orientationLB} have many useful, essentially immediate, corollaries which cover the construction of many graph families.

\begin{corollary} \label{cor::evenRegular}
Let $G$ be a $2d$-regular graph on the vertex set $[n]$. Then
$\tau(G, n) = \tau_{\text{lab}}(G, n) = d n$.
\end{corollary}

Let ${\mathcal H} = {\mathcal H}(n)$ denote the family of Hamiltonian cycles of $K_n$. It follows by Corollary~\ref{cor::evenRegular} that $\tau({\mathcal H}, n) = n$. For the original `with replacement' model, it was shown in~\cite{BHKPSS} that w.h.p. $\ln 2+\ln(1+\ln 2))-o(1) \leq \tau({\mathcal H}, n) /n \leq 3+o(1)$. Both bounds were subsequently improved in~\cite{GKMP}, where it was shown that w.h.p. $\ln 2+\ln(1+\ln 2))+10^{-8}-o(1) \leq \tau({\mathcal H}, n) /n \leq 2.07+\frac{4}{e^2}+o(1)$. Since $\ln 2+\ln(1+\ln 2))>1$, it follows that (for sufficiently large $n$) a Hamiltonian graph can be built in the `no replacement' model faster than in the `with replacement' model. We remark that there are graphs which can be built in the `with replacement' model faster than in the `no replacement' model. For example, \cite[Theorem 1.10]{BHKPSS} implies that  $\tau(K_4,n)=o(n)$ in the `with replacement' model, whereas it follows by Proposition~\ref{prop::orientationLB} that $\tau(K_4,n) > n$ in the `no replacement' model. It is not hard to augment this example to obtain spanning graphs which may be built in the `with replacement' model faster than in the `no replacement' model.

\begin{corollary}\label{cor:trees}
Let $G$ be a $d$-degenerate graph on the vertex set $[n]$. Then $e(G) \leq \tau(G, n) \leq \tau_{\textrm{lab}}(G, n) \leq d n$. In particular, in the special case where $e(G) = d n$, it holds that $\tau(G, n) = \tau_{\textrm{lab}}(G, n) = d n$. Another special case is when $T$ is a tree, and then $n-1 \leq \tau(T, n) \leq \tau_{\textrm{lab}}(T, n) \leq n$.
\end{corollary}

\begin{corollary}
Let $G$ be an arbitrary balanced \footnote{a graph $G$ is balanced if $e(G)/v(G) = \max \{e(H)/v(H) : \emptyset \neq H \subseteq G\}$.} graph with $m$ edges on the vertex set $[n]$. Then $m \leq \tau(G, n) \leq \tau_{\textrm{lab}}(G, n) \leq \lceil m/n \rceil n$. In particular, if $m/n$ is an integer, then $\tau(G,n) = \tau_{\textrm{lab}}(G,n) = m$.  
\end{corollary}

\begin{corollary}
Let $G \sim G(n,p)$, where $p = p(n) \geq \ln n/n$ and let $f : \mathbb{N} \to \mathbb{N}$ be a function satisfying $f(n) = \omega(n \sqrt{p})$. Then w.h.p.
$$
n^2 p/2 - f(n) \leq \tau(G, n) \leq \tau_{\textrm{lab}}(G, n) \leq n^2 p/2 + \sqrt{n^3 p \ln n} + n.
$$ 
\end{corollary}

\begin{proof}
Since $p \geq (1 + o(1)) \ln n/n$, it is well-known (see, e.g.,~\cite{Bollobas, FK, JLR}) that w.h.p. $e(G) \geq n^2 p/2 - f(n)$ (indeed, by Chernoff's bound, $\Pr\left(e(G) < n^2p/2 - f(n)\right) < e^{- f(n)^2/(n^2 p)} = o(1)$) and that w.h.p. $d_G(u) \leq n p + 2 \sqrt{n p \ln n}$ holds for every $u \in V(G)$ (indeed, for every $u \in V(G)$, it follows by Chernoff's bound that $\Pr(d_G(u) > np + 2\sqrt{np \ln n}) = o(1/n)$ and hence $\Pr(\exists u\in V(G): d_G(u) > np+2\sqrt{np\ln n})=o(1)$). It is then easy to find an orientation $D$ of the edges of $G$ such that $d^+_D(u) \leq \lceil d_G(u)/2 \rceil \leq n p/2 + \sqrt{n p \ln n} + 1$ holds for every $u \in V(G)$. It thus follows that w.h.p.
$$
n^2 p/2 - f(n) \leq e(G) \leq \tau(G, n) \leq \tau_{\textrm{lab}}(G, n) \leq n^2 p/2 + \sqrt{n^3 p \ln n} + n,
$$
where the second inequality holds by Proposition~\ref{prop::orientationLB} and the last inequality holds by Proposition~\ref{prop::orientationUB}.     
\end{proof}

It follows from all of the aforementioned corollaries that some properties which may still be interesting to study are the construction of odd regular graphs and the construction of a (not predetermined) graph from an interesting family, such as graphs of minimum degree $k$ or $k$-connected graphs (where $k$ is odd).

\section{An urn model} \label{sec::urn}

In this section we analyze an urn model which is somewhat reminiscent of Polya's urn model~\cite{Polya};
it will be used later on to prove Theorem~\ref{th::minDegk}, but may also have independent interest. 

We start with $n$ white balls in an urn. In each round, as long as there is at least one white ball in the urn, we remove one ball from the urn, chosen uniformly at random, and then if the urn still contains at least one white ball, we replace one white ball with one black ball. Let $T$ be the number of rounds until the process terminates (i.e., until there are no white balls left in the urn); clearly $T \leq n-1$. For every non-negative integer $i$, let $W_i$ be the number of white balls in the urn after exactly $\min \{i, T\}$ rounds. Clearly $W_0 = n$, $W_i = 0$ for every $i \geq T$, and for every $1 \leq i \leq T$, it holds that $W_i = W_{i-1} - 2$ if $W_{i-1} > 1$ and a white ball was chosen in the $i$th round (an event which occurs with probability $\frac{W_{i-1}}{n-(i-1)}$), and $W_i = W_{i-1} - 1$ otherwise.

Our main aim in this subsection is to prove that w.h.p. $T$ is very close to $(1 - 1/e) n$. 
We remark that this may be shown by using Wormald's differential equations method \cite{W1,W2}. For completeness, we present a simple direct argument.
We begin by estimating the expectation and variance of the $W_i$'s. Recall that $H_{\ell,m} = \sum_{k=\ell+1}^m \frac{1}{k}$ for every $\ell \leq m$.

\begin{claim}
For every $0 < j < n$ it holds that
\begin{align}
{\mathbb E}(W_j) &\geq (n-j) \left(1 - H_{n-j-1,n-1}\right) \label{eq:E_lower} \\
{\mathbb E}(W_j) &\leq (n-j) \left(1 - \Pr(T>j) H_{n-j-1,n-1}\right) \label{eq:E_upper} \\
{\rm Var}(W_j) &\leq \textstyle\frac{5}{4} j \label{eq:Var}
\end{align}
\end{claim}

\begin{proof}
Observe that for every $0 < i < n$, it follows from the definition of the $W_i$'s that
\begin{equation*}
{\mathbb E}(W_i \mid W_{i-1})=\begin{cases}
0 & W_{i-1} \leq 1 \\
W_{i-1} - 1 - \frac{W_{i-1}}{n-(i-1)} & W_{i-1} > 1.
\end{cases}
\end{equation*}
Hence
\begin{equation}\label{eq:conditional_expectation}
\frac{1}{n-i}{\mathbb E}(W_i \mid W_{i-1}) = \frac{1}{n-(i-1)} \tilde{W}_{i-1} - \frac{1}{n-i},
\end{equation}
where
\begin{equation*}
\tilde{W}_{i-1}=\begin{cases}
\frac{n-(i-1)}{n-i} & W_{i-1} \leq 1 \\
W_{i-1} & W_{i-1} > 1.
\end{cases}
\end{equation*}
Note that
\begin{equation}\label{eq:tilde}
{\mathbb E} (\tilde{W}_{i-1}) = {\mathbb E}(W_{i-1}) + \frac{n-(i-1)}{n-i} \Pr(W_{i-1} = 0) + \frac{1}{n-i} \Pr(W_{i-1} = 1).
\end{equation}
Therefore
\begin{equation} \label{eq::Etilde}
{\mathbb E} (\tilde{W}_{i-1}) \leq {\mathbb E}(W_{i-1}) + \frac{n-(i-1)}{n-i} \Pr(W_{i-1} \leq 1) \leq {\mathbb E}(W_{i-1}) + \frac{n-(i-1)}{n-i} \Pr(T\leq i).
\end{equation}
It follows that
\begin{align*}
\frac{1}{n-i} {\mathbb E}(W_i) &= {\mathbb E}\left(\frac{1}{n-i}{\mathbb E}(W_i \mid W_{i-1})\right) = \frac{1}{n-(i-1)} {\mathbb E}(\tilde{W}_{i-1}) - \frac{1}{n-i} \\ &\leq \frac{1}{n-(i-1)}{\mathbb E}(W_{i-1}) + \frac{1}{n-i} \Pr (T \leq i) - \frac{1}{n-i} = \frac{1}{n-(i-1)} {\mathbb E}(W_{i-1}) - \frac{\Pr(T > i)}{n-i},
\end{align*}
where the first equality holds by the law of total expectation, the second equality holds by~\eqref{eq:conditional_expectation}, and the inequality holds by~\eqref{eq::Etilde}. Therefore, for every $0 < j < n$, it holds that
\begin{equation*}
\frac{1}{n-j} {\mathbb E}(W_j) \leq \frac{1}{n} {\mathbb E}(W_0) - \sum_{i=1}^j \frac{\Pr(T>i)}{n-i} \leq 1 - \Pr(T>j) \sum_{i=1}^j \frac{1}{n-i},
\end{equation*}
which proves~\eqref{eq:E_upper}. Similarly, \eqref{eq:E_lower} follows upon observing that 
\begin{equation}\label{eq:E_tilde}
{\mathbb E} (\tilde{W}_{i-1}) \geq {\mathbb E}(W_{i-1}).
\end{equation}
holds by~\eqref{eq:tilde}.

We proceed to prove~\eqref{eq:Var}. For every $0 < i < n$ it holds that
\begin{equation*}
{\rm Var}(W_i\mid W_{i-1}) = {\rm Var}(W_i - W_{i-1} \mid W_{i-1})=\begin{cases}
0 & W_{i-1} \leq 1 \\
\frac{W_{i-1}}{n-(i-1)} \left(1 - \frac{W_{i-1}}{n-(i-1)}\right) & W_{i-1} > 1.
\end{cases}
\end{equation*}
Since the maximum of $x (1-x)$ is $1/4$, it follows that ${\rm Var}(W_i \mid W_{i-1}) \leq 1/4$ and thus
\begin{equation*}
{\mathbb E}\left({\rm Var}(W_i \mid W_{i-1})\right) \leq \textstyle\frac{1}{4}.
\end{equation*}
Moreover, using~\eqref{eq:E_tilde}, we obtain
\begin{equation} \label{eq:Var_tilde}
{\rm Var} (\tilde{W}_{i-1}) \leq {\rm Var}(W_{i-1}) + {\mathbb E}\left(\tilde{W}_{i-1}^2 - W_{i-1}^2 \right) \leq {\rm Var}(W_{i-1}) + \left(\frac{n-(i-1)}{n-i}\right)^2. 
\end{equation}
Therefore 
\begin{align*}
{\rm Var}\left({\mathbb E}(W_i \mid W_{i-1})\right) &= {\rm Var}\left(\frac{n-i}{n-(i-1)} \tilde{W}_{i-1} - 1 \right) = \left(\frac{n-i}{n-(i-1)}\right)^2 {\rm Var}(\tilde{W}_{i-1}) \\
&\leq \left(\frac{n-i}{n-(i-1)}\right)^2 {\rm Var}(W_{i-1}) + 1 \leq {\rm Var}(W_{i-1}) + 1,
\end{align*}
where the first equality holds by~\eqref{eq:conditional_expectation} and the first inequality holds by~\eqref{eq:Var_tilde}.

It then follows by the law of total variance that 
\begin{equation*}
{\rm Var}(W_i) = {\rm Var}\left({\mathbb E}(W_i \mid W_{i-1})\right) + {\mathbb E}\left({\rm Var}(W_i \mid W_{i-1})\right) < {\rm Var}(W_{i-1}) + \textstyle\frac{5}{4}.
\end{equation*}
Noting that ${\rm Var}(W_0) = 0$, this proves~\eqref{eq:Var}.
\end{proof}

\begin{proposition} \label{prop::urn1concentration}
Let $m_0 = \lfloor (1 - 1/e) n \rfloor$ and let $\alpha(n)$ be a positive integer smaller than $m_0$ (in particular, $n\geq 4$). Then
\begin{equation} \label{eq:polya_lower}
\Pr\left(T < m_0 - \alpha(n) \right) < \frac{6n}{(\alpha(n))^2} 
\end{equation}
and
\begin{equation} \label{eq:polya_upper}
\Pr\left(T > m_0 + 36 \alpha(n) + \frac{12n}{(\alpha(n))^2}\right) < \frac{6n}{(\alpha(n))^2}. 
\end{equation}
\end{proposition}

\begin{proof}
Denote $m := m_0 - \alpha(n)$.
It follows by~\eqref{eq:harmonic1} and~\eqref{eq:harmonic2} that
\begin{align} 
H_{n-m-1,n-1}&=H_{n-m_0,n}-H_{n-m_0,n-m-1}-\frac{1}{n} \leq \ln\frac{n}{n-m_0}-\frac{m_0-m-1}{n-m-1}-\frac{1}{n}\nonumber\\
&<\ln e-\frac{m_0-m-1}{n}-\frac{1}{n}=1- \frac{\alpha(n)}{n}. \label{eq::smallFirstSum}
\end{align}
Therefore
\begin{equation} \label{eq::EWmLarge}
{\mathbb E}(W_m) \geq (n - m) \left(1 -H_{n-m-1,n-1}\right) \geq \left(\frac{n}{e} + \alpha(n)\right) \frac{\alpha(n)}{n} \geq \frac{1}{e} \alpha(n),
\end{equation}
where the first inequality holds by~\eqref{eq:E_lower} and the second inequality holds by~\eqref{eq::smallFirstSum}. Thus 
\begin{equation}\label{eq:chebyshev}
\Pr\Big(|W_m - {\mathbb E}(W_m)| \geq {\mathbb E}(W_m)\Big) \leq \frac{{\rm Var}(W_m)}{({\mathbb E}(W_m))^2} \leq \frac{\frac{5}{4} e^2 m}{(\alpha(n))^2} \leq \frac{\frac{5}{4} e (e-1) n}{(\alpha(n))^2} < \frac{6n}{(\alpha(n))^2},
\end{equation}
where the first inequality holds by Chebyshev's inequality and the second inequality holds by~\eqref{eq:Var} and by~\eqref{eq::EWmLarge}. Therefore
\begin{equation} \label{eq::StrongForm}
\Pr\left(T \leq m\right) = \Pr\left(W_m = 0 \right) \leq \Pr\Big(|W_m - {\mathbb E}(W_m)| \geq {\mathbb E}(W_m) \Big) < \frac{6n}{(\alpha(n))^2},
\end{equation}
which proves~\eqref{eq:polya_lower}. We proceed to prove~\eqref{eq:polya_upper}. 
It follows by~\eqref{eq:harmonic1} and~\eqref{eq:harmonic2} that
\begin{align} 
H_{n-m-1,n-1} &= H_{n-m_0-2,n-1} - H_{n-m_0-2,n-m-1} \geq \ln\frac{n}{n-m_0-1} - \frac{m_0-m+1}{n-m_0-1}\nonumber\\
&> \ln e - \frac{m_0-m+1}{n/e-1} = 1-\frac{\alpha(n) + 1}{n/e - 1} \geq 1-\frac{2\alpha(n)}{n/e - n/4} > 1 - \frac{18 \alpha(n)}{n}. \label{eq::firstSumLarge}
\end{align}
Therefore
\begin{align} \label{eq::EWmUB}
{\mathbb E}(W_m) &\leq  (n - m) \left(1 - \Pr(T > m) H_{n-m-1,n-1}\right) \nonumber \\
&< n \left(1 - \left(1 - \frac{6}{(\alpha(n))^2}\right) \left(1 - \frac{18 \alpha(n)}{n}\right)\right) < 18 \alpha(n) + \frac{6n}{(\alpha(n))^2},
\end{align}
where the first inequality holds by~\eqref{eq:E_upper} and the second inequality holds by~\eqref{eq::StrongForm} and~\eqref{eq::firstSumLarge}.

Denote $m_1 := \left\lceil 18 \alpha(n) + \frac{6n}{(\alpha(n))^2} \right\rceil$. Then   
\begin{align*}
\Pr\left(T \geq m + 2 m_1 \right) &= \Pr\left(W_{m + 2 m_1 - 1} \geq 1 \right) \leq \Pr\left(W_m \geq 2 m_1 \right) \leq
\Pr\big(W_m \geq 2 {\mathbb E}(W_m)\big) \\
&\leq \Pr\Big(|W_m - {\mathbb E}(W_m)| \geq {\mathbb E}(W_m)\Big) < \frac{6n}{(\alpha(n))^2}.
\end{align*}
where the first inequality holds since $W_{m + 2 m_1 - 1} \leq W_m - (2 m_1 - 1)$, the second inequality holds by~\eqref{eq::EWmUB} and the last inequality holds by~\eqref{eq:chebyshev}. This proves~\eqref{eq:polya_upper}.
\end{proof}

\section{Minimum degree} \label{sec::MinDeg}
In this section we consider minimum degree games. Let ${\mathcal D}_d = {\mathcal D}_d(n)$ be the family of $n$-vertex simple graphs with minimum degree at least $d$. Note that $\tau(\mathcal{D}_d, n) = \tau_{\text{lab}}(\mathcal{D}_d, n)$ for every $d$ and every $n$.

\begin{theorem} \label{th::minDegk}
Let $d \leq n-1$ be a positive integer and let $f : \mathbb{N} \to \mathbb{N}$ be a function satisfying $f(n) = \omega(\sqrt{n})$. 
\begin{description}
\item [(i)] If $d$ is even and $k = d/2$, then $\tau(\mathcal{D}_d, n) = k n$;

\item [(ii)] If $d$ is odd and $k = (d-1)/2$, then w.h.p. it holds that
$$
\left(k + 1 - 1/e \right) n - f(n) \leq \tau(\mathcal{D}_d, n) \leq \left(k + 1 - 1/e \right) n + f(n) + 2k,
$$
where the upper bound holds under the additional assumption that $d = o(n)$.
\end{description}
\end{theorem}

We remark that for the `with replacement' model, it was shown in \cite{BHKPSS} that w.h.p. $\tau(\mathcal{D}_d, n) /n=h_d+o(1)$, where $h_d$ is a constant that was presented by Wormald in \cite{W2}, and may be computed by using his differential equations method \cite{W1,W2}. 
The first few $h_d$’s were explicitly calculated; it was shown in \cite{W2, KKLR} that $h_1=\ln 2$ and $h_2=\ln 2+\ln(1+\ln 2)$, and in \cite{KKLR} that $h_3=\ln\left((\ln 2)^2+2(1+\ln 2)(1+\ln(1+\ln 2))\right)$. 
In particular, $h_1 > 1 - 1/e$, $h_2 > 1$, and $h_3 > 2 - 1/e$. Hence, for $d \in \{1, 2, 3\}$, a graph with minimum degree $d$ can be built in the `no replacement' model faster than in the `with replacement' model. It would be interesting to determine if this is true in general.

We will first prove the special case $d=1$ of Theorem~\ref{th::minDegk}. 

\begin{proposition} \label{prop::minDeg1}
Let $f : \mathbb{N} \to \mathbb{N}$ be a function satisfying $f(n) = \omega(\sqrt{n})$. Then w.h.p. it holds that
$$
\left(1 - 1/e \right) n - f(n) \leq \tau(\mathcal{D}_1, n) \leq \left(1 - 1/e \right) n + f(n).
$$
\end{proposition}

\begin{proof}
Starting with the upper bound, consider the following strategy, which we denote by $\mathcal{S}_0$: as long as there are isolated vertices in his graph, in every round Builder connects the vertex he is offered to an arbitrary isolated vertex (if he is offered the last isolated vertex, then he connects it to an arbitrary vertex). 

Now, consider an urn containing $n$ white balls, each one representing one vertex. Whenever Builder is offered a vertex $u$, the ball corresponding to $u$ is removed from the urn. Moreover, if Builder then connects $u$ to $v$, the ball corresponding to $v$ is replaced by a black ball; note that, by the description of $\mathcal{S}_0$, before it was replaced, the ball corresponding to $v$ was white (unless $u$ was the last isolated vertex). This is precisely the urn model described in Section~\ref{sec::urn}, where white balls represent isolated vertices, black balls represent non-isolated vertices that were not yet offered, and the process terminates precisely when Builder's graph has positive minimum degree. It thus follows by Proposition~\ref{prop::urn1concentration} that w.h.p. $\tau(\mathcal{D}_1, n) \leq \left(1 - 1/e \right) n + f(n)$.

For the lower bound, our main goal is to prove the following claim.

\begin{claim} \label{cl::S0Optimal}
$\mathcal{S}_0$ is optimal for $(\mathcal{D}_1, n)$.
\end{claim}

\begin{proof}
For a positive integer $i$, we say that a strategy $\mathcal{S}$ is \textit{$i$-natural} if for every $1 \leq j \leq i$, in the $j$th round, if after Builder is offered a vertex $u_j$ there is still an isolated vertex $v \neq u_j$ in his graph, $\mathcal{S}$ instructs Builder to connect $u_j$ to an isolated vertex. A strategy is said to be \textit{natural} if it is $i$-natural for every $i$. Noting that $\mathcal{S}_0$ is natural, that any two natural strategies are equivalent (in the sense that each one dominates the other), that any strategy is $1$-natural, and that domination is a transitive relation, in order to prove the claim it suffices to prove that, for any positive integer $i$ and any $i$-natural strategy $\mathcal{S}$, there exists an $(i+1)$-natural strategy $\mathcal{S}'$ which dominates $\mathcal{S}$. Let $\mathcal{S}$ be an arbitrary $i$-natural strategy. If $\mathcal{S}$ is $(i+1)$-natural, then set $\mathcal{S}' = \mathcal{S}$ (note that domination is a reflexive relation); assume then that it is not. Define $\mathcal{S}'$ as follows: 
\begin{description}
\item [$1 \leq j \leq i$:] In the $j$th round, Builder plays as instructed by $\mathcal{S}$.

\item [$j = i+1$:] Let $u_{i+1}$ be the vertex Builder is offered in round $i+1$ and let $v_{i+1}$ be the vertex he connects it to when playing according to $\mathcal{S}$. Instead, Builder connects $u_{i+1}$ to an arbitrary isolated vertex $y$ (such a vertex exists as otherwise $\mathcal{S}$ would be $(i+1)$-natural contrary to our assumption).

\item [$j > i+1$:] In the $j$th round, Builder plays as instructed by $\mathcal{S}$ under the assumption that in round $i+1$ he claimed the edge $u_{i+1} v_{i+1}$ (whenever he has to claim an edge he has already claimed, he claims an arbitrary edge instead).        
\end{description} 
Clearly $\mathcal{S}'$ is $(i+1)$-natural. Moreover, it follows from the description of $\mathcal{S}'$ that, at any point during the game $(\mathcal{D}_1, n)$, the set of isolated vertices in Builder's graph when he plays according to $\mathcal{S}$ contains the set of isolated vertices in his graph when he plays according to $\mathcal{S}'$. Hence, $\mathcal{S}'$ dominates $\mathcal{S}$.   
\end{proof}

We conclude that w.h.p. $\tau(\mathcal{D}_1, n) = \tau(\mathcal{S}_0) \geq \left(1 - 1/e \right) n - f(n)$, where the equality holds by Claim~\ref{cl::S0Optimal} and the inequality holds w.h.p. by Proposition~\ref{prop::urn1concentration}.           
\end{proof}

Theorem~\ref{th::minDegk} is an easy corollary of the following two lemmas. The first lemma will be used again in Section~\ref{sec::EdgeCon}.

\begin{lemma} \label{lem::edgeDisjoint}
Let $G$ be a graph on the vertex set $[n]$ with maximum degree $\Delta = o(n)$, let $\mathcal{D}_{1,G}$ be the family of positive minimum degree graphs on the vertex set $[n]$ which are edge disjoint from $G$, and let $f : \mathbb{N} \to \mathbb{N}$ be a function satisfying $f(n) = \omega(\sqrt{n})$.
Then w.h.p. it holds that
$$
\tau(\mathcal{D}_{1,G}, n) \leq \left(1 - 1/e \right) n + f(n) + \Delta.
$$
\end{lemma}

\begin{proof}
We may assume that $f(n)=o(n)$. 
Hence, we may assume that $(1 - 1/e) n + f(n) + \Delta - 1 < 2n/3$.
At any point during the game, let $A$ denote the set of isolated vertices in Builder's graph $H$ and let $B = \{u \in [n] : uv \in E(G) \textrm{ for every } v \in A\}$; note that $|B| \leq \Delta$ holds throughout the process. In order to build $H$, Builder follows the strategy $\mathcal{S}_0$ which is described in the proof of Proposition~\ref{prop::minDeg1}. Since, when playing according to $\mathcal{S}_0$, Builder is allowed to connect the vertex he is offered to any isolated (in $H$) vertex, he connects the vertex $u$ he is offered to some $v \in A \setminus N_G(u)$. Whenever this is not possible, even though $A \neq \emptyset$, Builder claims an arbitrary edge $uz$ and the corresponding round is declared a \emph{failure}. 

Let $i$ denote the index of the round at the beginning of which $|A| \leq \Delta$ first occurs; observe that there are no failures prior to the $i$th round.
Hence, it follows by Proposition~\ref{prop::minDeg1} that w.h.p. $i \leq \left(1 - 1/e \right) n + f(n) - (\Delta - 1)/2$. For every integer $i \leq j < i + (3\Delta - 1)/2$, Let $u_j$ be the vertex Builder is offered in round $j$. If round $j$ is a failure, we must have $u_j \in B$. 
The probability of this event is $\frac{|B|}{n-(j-1)} < \frac{\Delta}{n/3}$, where the inequality holds since $|B| \leq \Delta$ and $j-1 < i + (3\Delta - 1)/2 - 1 \leq (1 - 1/e) n + f(n) + \Delta - 1 < 2n/3$. 
Let $F$ be the number of indices $i \leq j < i + (3\Delta - 1)/2$ for which $u_j \in B$. 
Then 
$\mathbb{E}(F) \leq 2\Delta \frac{\Delta}{n/3} = o_n(\Delta)$, 
where the equality holds by our assumption that $\Delta = o(n)$. It then follows by Markov's inequality that w.h.p. $F \leq (\Delta - 1)/2$. This leaves at least $\Delta$ indices $i \leq j < i + (3\Delta - 1)/2$ for which the $j$th round is not a failure, each such round decreasing the size of $|A|$ by at least 1 (unless $A = \emptyset$ and Builder has already won). 
Hence, w.h.p. building $H$ requires at most $i + (3\Delta - 1)/2 \leq \left(1 - 1/e \right) n + f(n) + \Delta$ rounds.
\end{proof}

\begin{lemma} \label{lem::degreeSequence}
Let $f : \mathbb{N} \to \mathbb{N}$ be a function satisfying $f(n) = \omega(\sqrt{n})$. Let $\bar{h} = (h_1, \ldots, h_n)$ be a vector of non-negative integers satisfying $\sum_{i=1}^n h_i \geq n$. Let $(\mathcal{D}_{\bar{h}}, n)$ be the game which Builder wins as soon as, for every $1 \leq i \leq n$, the degree of vertex $i$ in his graph is at least $h_i$. Then w.h.p. it holds that
$$
\tau(\mathcal{D}_{\bar{h}}, n) \geq \left(1 - 1/e \right) n - f(n).
$$
\end{lemma}

\begin{proof}
Let $\bar{h} = (h_1, \ldots, h_n)$, let $\mathcal{S}$ be an arbitrary strategy for Builder in $(\mathcal{D}_{\bar{h}}, n)$ and let $A = \{1 \leq i \leq n : h_i = 0\}$. We will prove by induction on $|A|$ that w.h.p. $\tau(\mathcal{S}) \geq \left(1 - 1/e \right) n - f(n)$. The induction basis, $|A| = 0$, follows by Proposition~\ref{prop::minDeg1}. For the induction step, assume that $1 \leq |A| \leq n$. Let $u$ be an arbitrary vertex of $A$ and let $v$ be an arbitrary vertex such that $h_v \geq 2$; such vertices exist since $|A| \geq 1$ and $\sum_{i=1}^n h_i \geq n$. Consider the following Builder's strategy $\mathcal{S}'$: Builder follows $\mathcal{S}$ until the first round $i$ in which he is offered some vertex $z$ which he is instructed to connect to $v$ (since $h_v \geq 2$, if no such round exists, then $\tau(\mathcal{S}) > n$ and we are done). In round $i$, Builder claims the edge $zu$ instead if it is free; otherwise he claims $zv$ as instructed. In every subsequent round, he follows $\mathcal{S}$ under the assumption that he claimed $zv$ in round $i$. As soon as he wins $(\mathcal{D}_{\bar{h}}, n)$ (under the assumption that he claimed $zv$ in round $i$) his graph $G$ satisfies $d_G(u) \geq 1$, $d_G(v) \geq h_v - 1$, and $d_G(i) \geq h_i$ for every $i \in [n] \setminus \{u,v\}$. 
Hence, $\tau(\mathcal{S}) \geq \tau(\mathcal{S}')$.    
It follows by the induction hypothesis that w.h.p. $\tau(\mathcal{S}') \geq \left(1 - 1/e \right) n - f(n)$. Therefore, w.h.p. $\tau(\mathcal{S}) \geq  \left(1 - 1/e \right) n - f(n)$ holds. Since $\mathcal{S}$ was arbitrary, it follows that $\tau(\mathcal{D}_{\bar{h}}, n) \geq \left(1 - 1/e \right) n - f(n)$ as claimed.    
\end{proof}

\begin{proof} [Proof of Theorem~\ref{th::minDegk}]
The lower bound in (i) is trivial and the upper bound is an immediate consequence of Corollary~\ref{cor::evenRegular}. 

In order to prove the upper bound in (ii) we present a strategy for Builder; it is divided into two stages. 
In the first stage, Builder constructs an arbitrary $2k$-regular graph $G$; by Proposition~\ref{prop::orientationUB} this can be done in $k n$ rounds. 
In the second stage, Builder constructs a graph $H$ with positive minimum degree which is edge disjoint from $G$. By Lemma~\ref{lem::edgeDisjoint}, this can be done w.h.p. within $\left(1 - 1/e \right) n + f(n) + 2k$ additional rounds.
We conclude that w.h.p. $\tau(\mathcal{D}_d, n) \leq \left(k + 1 - 1/e \right) n + f(n) + 2k$.           

It remains to prove the lower bound in (ii). Let $\mathcal{S}$ be a strategy for Builder in $(\mathcal{D}_d, n)$. Observe that $\tau(\mathcal{S}) \geq d n/2 > k n$. Let $G$ denote Builder's graph after following $\mathcal{S}$ for $k n$ rounds. For every $1 \leq i \leq n$, let $h_i = \max \{0, 2k+1 - d_G(i)\}$; note that $h_i$ is a non-negative integer for every $1 \leq i \leq n$ and $\sum_{i=1}^n h_i \geq n$. In order to build a graph with minimum degree at least $2k+1$, Builder has to build a graph $H$ such that $E(H) \cap E(G) = \emptyset$ and $d_H(i) \geq h_i$ for every $1 \leq i \leq n$. By Lemma~\ref{lem::degreeSequence} w.h.p. this requires at least $\left(1 - 1/e \right) n - f(n)$ rounds. Therefore, for any strategy $\mathcal S$, w.h.p. $\tau(\mathcal{S}) \geq \left(k + 1 - 1/e \right) n - f(n)$. 
We conclude that w.h.p. $\tau(\mathcal{D}_d, n) \geq \left(k + 1 - 1/e \right) n - f(n)$.
\end{proof}

\section{Perfect matching} \label{sec::PM}
In this section we consider perfect matching games. Throughout this section we assume $n$ to be an even integer. Let $\mathcal{PM} = \mathcal{PM}(n)$ be the family of all perfect matchings on the vertex set $[n]$, and let $M_0$ be the particular perfect matching whose set of edges is $\{(2i-1, 2i) : 1 \leq i \leq n/2\}$. Since $\mathcal{PM} = \{M_0\}_{\text{iso}}$, it follows that $\tau(\mathcal{PM},n) = \tau_{\text{lab}}(\mathcal{PM},n) = \tau(M_0,n)$.
\begin{theorem} \label{th::PM}
$\tau(\mathcal{PM}, n) = \tau_{\text{lab}}(M_0,n)$. Moreover, given any functions $f, g : \mathbb{N} \to \mathbb{N}$ such that $f(n) = \omega \left(\sqrt{n}\right)$ and $g(n) = o\left(\sqrt{n}\right)$, w.h.p.
$$
n - f(n) \leq \tau_{\text{lab}}(M_0,n) \leq n -g(n).
$$
\end{theorem}

We remark that for the `with replacement' model, it was shown in \cite{BHKPSS} that w.h.p. $\ln 2+o(1) \leq \tau(\mathcal{PM}, n) /n \leq 1 + 2/e + o(1)$. Both  bounds were subsequently improved in~\cite{GMP}, where it was shown that w.h.p. $0.93261 + o(1) \leq \tau(\mathcal{PM}, n) /n \leq 1.20524 + o(1)$.

\begin{proof}[Proof of Theorem \ref{th::PM}]
For every $i\in[n]$ let $\mu(i) := i + (-1)^{i+1}$ denote the vertex that is matched to $i$ in the matching $M_0$. 
Consider the following concrete strategy, which we denote by $\mathcal{S}_0$: until his graph first admits the perfect matching $M_0$, Builder maintains a graph in which every connected component contains a unique vertex which was not yet offered; we will mark this vertex as being \emph{active}. Moreover, if $u$ is the active vertex of a connected component $C$, then $\mu(u)$ is active as well if and only if the number of vertices in $C$ is odd. 
These properties clearly hold before the game starts. Assume that they are satisfied immediately after $i-1$ rounds for some $1 \leq i < n$ and that Builder's graph at this point does not yet admit the perfect matching $M_0$. Let $u_i$ be the vertex Builder is offered in round $i$. 
If the vertex $\mu(u_i)$ is active, Builder connects $u_i$ to $\mu(u_i)$.
Otherwise, he connects $u_i$ to  a vertex in some other connected component (note that the graph must contain at least two odd connected components at this point. Indeed, for some $1 \leq j \leq n/2$, the pair $\{2j-1, 2j\}$ was not yet connected by Builder; hence both $2j-1$ and $2j$ are active and thus reside in two different odd components). It is straightforward to verify that Builder can indeed follow this strategy until his graph admits the perfect matching $M_0$.

It is evident that, following this strategy, Builder completes the matching $M_0$ as soon as he is offered at least one vertex of $\{2j-1, 2j\}$ for every $1 \leq j \leq n/2$, but not sooner. 
It thus follows by Claim~\ref{claim:match_comp} that w.h.p. 
$n - f(n) \leq\tau({\mathcal S}_0) \leq n - g(n)$.

Let $(\mathcal{EC}, n)$ be the game Builder wins as soon as every connected component in his graph is even. 
In order to complete the proof, we will prove that ${\mathcal S}_0$ is in fact an optimal strategy for this game.

\begin{claim} \label{cl::S0OptimalMatching}
$\mathcal{S}_0$ is optimal for $(\mathcal{EC}, n)$.
\end{claim}

\begin{proof}
For a nonnegative integer $i$, we say that a strategy $\mathcal{S}$ is \textit{$i$-good} if for every $1 \leq j \leq i$, upon being offered $u_j$ in the $j$th round, $\mathcal{S}$ instructs Builder to act as follows: if $\mu(u_j)$ is still active, Builder is to connect $u_j$ to $\mu(u_j)$; otherwise, he should connect $u_j$ to an arbitrary vertex in some other connected component.
A strategy is said to be \textit{good} if it is $i$-good for every $i$. 
Noting that $\mathcal{S}_0$ is good, that any two good strategies are equivalent (in the sense that each one dominates the other), that any strategy is $0$-good, and that domination is a transitive relation, in order to prove the claim it suffices to prove that, for any positive integer $i$ and any $(i-1)$-good strategy $\mathcal{S}$, there exists an $i$-good strategy $\mathcal{S}'$ which dominates $\mathcal{S}$. Let $\mathcal{S}$ be an arbitrary $(i-1)$-good strategy. If $\mathcal{S}$ is $i$-good, then set $\mathcal{S}' = \mathcal{S}$ (note that domination is a reflexive relation); assume then that it is not. Define $\mathcal{S}'$ as follows: 
\begin{description}
\item [$1 \leq j < i$:] In the $j$th round, Builder plays as instructed by $\mathcal{S}$.

\item [$j = i$:] Let $u_i$ be the vertex Builder is offered in round $i$ and let $v_i$ be the vertex he connects it to when playing according to $\mathcal{S}$. Let $C$ be the connected component containing $u_i$.
Since $\mathcal{S}$ is not $i$-good, there can be only two cases. 
\begin{description}
\item [(a)] 
$\mu(u_i)$ is active and $v_i \neq \mu(u_i)$.
Instead, Builder connects $u_i$ to $\mu(u_i)$.

\item [(b)] 
$\mu(u_i)$ is not active and $v_i \in C$.
Instead, Builder connects $u_i$ to a vertex in some other connected component (note that the graph contains at least two odd connected components at this point). 

\end{description}

\item [$j > i$:] In the $j$th round, Builder plays as instructed by $\mathcal{S}$ under the assumption that in round $i$ he claimed the edge $u_i v_i$ (whenever he has to claim an edge he has already claimed, he claims an arbitrary edge instead).        
\end{description} 

Clearly $\mathcal{S}'$ is $i$-good. Moreover, it follows from the description of $\mathcal{S}'$ that, at any point during the game $(\mathcal{EC}, n)$, the number of odd components in Builder's graph when he plays according to $\mathcal{S}$ is at least as large as the number of odd components in Builder's graph when he plays according to $\mathcal{S}'$. Hence, $\mathcal{S}'$ dominates $\mathcal{S}$.   
\end{proof}
                
Since Builder's graph cannot admit a perfect matching as long as it contains an odd component, we conclude that $\tau(\mathcal{PM}, n) \geq \tau(\mathcal{EC}, n) = \tau(\mathcal{S}_0)$.
Since, obviously, $\tau(\mathcal{PM}, n)\leq \tau_{\text{lab}}(M_0,n)\leq\tau(\mathcal{S}_0)$, 
it follows that
$\tau(\mathcal{PM}, n)= \tau_{\text{lab}}(M_0,n)=\tau(\mathcal{S}_0)$,
which concludes the proof of the Theorem.
\end{proof}

\section{Building regular graphs} \label{sec::oddRegularGraphs}

In this section we consider the construction of regular graphs. Let $G$ be a $d$-regular graph on $n$ vertices. If $d$ is even, then $\tau(G, n) = \tau_{\text{lab}}(G, n)  = d n/2$ holds by Corollary~\ref{cor::evenRegular}. If, on the other hand, $d$ is odd, then 
\begin{equation} \label{eq::oddRegularSimpleBounds}
((d + 1)/2 - 1/e - o(1)) n \leq \tau({\mathcal D}_d, n) \leq \tau(G, n) \leq \tau_{\text{lab}}(G, n) \leq (d+1) n/2
\end{equation}
holds by Proposition~\ref{prop::orientationUB} and by Theorem~\ref{th::minDegk}(ii). The following result shows that the upper bound in~\eqref{eq::oddRegularSimpleBounds} is asymptotically tight for $\tau_{\text{lab}}(G, n)$. 

\begin{theorem} \label{th::oddRegularGraph}
Let $n$ be an even integer, let $1 \leq k \leq n/2 - 1$ be an integer and let $f : \mathbb{N} \to \mathbb{N}$ be a function satisfying $f(n) = \omega(\sqrt{n})$. Let $G$ be a $(2k+1)$-regular graph on the vertex set  $[n]$. Then w.h.p. $\tau_{\text{lab}}(G, n) \geq (k+1) n - f(n)$. 
\end{theorem}

For the complete graph, the game $(K_n, n)$ is obviously the same as the game  $(K_n, n)_{\text{lab}}$. Hence, while $\tau(K_n, n) = e(K_n) = \binom{n}{2}$ holds by Corollary~\ref{cor::evenRegular} whenever $n$ is odd, Theorem~\ref{th::oddRegularGraph} yields a slightly better lower bound on $\tau(K_n, n)$ when $n$ is even.
\begin{corollary} \label{cor::Kn}
Let $n$ be an even integer and let $f : \mathbb{N} \to \mathbb{N}$ be a function satisfying $f(n) = \omega(\sqrt{n})$. Then w.h.p. $\tau(K_n, n) \geq n^2/2 - f(n)$. 
\end{corollary}

The main ingredient in our proof of Theorem~\ref{th::oddRegularGraph} is the following lemma.

\begin{lemma} \label{lem::GeneralizedPerfectMatching}
Let $n$ be an even integer and let $G$ be a graph on the vertex set $[n]$ with $n/2$ edges in which every connected component is either a cycle or it contains a vertex of degree $1$. Let $f : \mathbb{N} \to \mathbb{N}$ be a function satisfying $f(n) = \omega(\sqrt{n})$. Then w.h.p. $\tau_{\text{lab}}(G, n) \geq n - f(n)$.    
\end{lemma}

\begin{proof}
Fix a positive integer $r$.
A permutation $\pi \in S_n$ is said to be \emph{good} (with respect to $G$ and $r$) if for every connected component $C = (V_C, E_C)$ of $G$, it holds that $|\{\pi(1), \ldots, \pi(n-r)\} \cap V_C| \geq |E_C|$. 
If Builder is able to construct a graph containing $G$ within $n-r$ rounds when the vertices are offered according to $\pi \in S_n$, then $\pi$ must obviously be a good permutation.
Let $S_G \subseteq S_n$ denote the set of all good permutations with respect to $G$ and $r$. 
It follows from the aforementioned observation that $\Pr(\tau_{\text{lab}}(G, n) \leq n - r) \leq |S_G|/n!$. Therefore, in order to complete the proof of the lemma, it suffices to prove that if $r=\omega(\sqrt{n})$ then $|S_G|/n! = o(1)$. 

Assume first that $G$ is a perfect matching. Assume without loss of generality that $G$ is the matching $M_0 = \{(2i-1, 2i) : 1 \leq i \leq n/2\}$ described in Section \ref{sec::PM}. By definition, $\pi \in S_n$ is good with respect to this graph $G$ if and only if 
$\{\pi(1), \ldots, \pi(n-r)\} \cap \{2i-1, 2i\} \neq\emptyset$
for every $1 \leq i \leq n/2$. 
It follows by Claim~\ref{claim:match_comp} that if $r = \omega(\sqrt{n})$, then $|S_G|/n! = o(1)$.

Now, let $G$ be an arbitrary graph which satisfies the conditions of the lemma, but is not a perfect matching. We construct a sequence of graphs $G_0, G_1, \ldots, G_t$ such that the following properties hold
\begin{description}
\item [(1)] $G_0 = G$;

\item [(2)] $G_t$ is a perfect matching;

\item [(3)] $V(G_j) = V(G)$ and $e(G_j) = n/2$ for every $0 \leq j \leq t$;

\item [(4)] For every $0 \leq j \leq t$, every connected component of $G_j$ is either a cycle or it contains a vertex of degree $1$
\end{description}
as follows. Assume that we have already defined $G_0, G_1, \ldots, G_i$ and that $G_i$ is not a perfect matching. Let $C$ be a connected component of $G_i$ with vertices $w_1, \ldots, w_{\ell}$ for some $\ell \geq 3$; since $G_i$ is not a perfect matching, such a component $C$ exists by Property (3). Assume first that $C$ is a cycle. Let $x_1, \ldots, x_{\ell}$ be isolated vertices of $G_i$; such vertices exist by Property (3). Let $G_{i+1} = (G_i \setminus C) \cup \{w_i x_i : 1 \leq i \leq \ell\}$. Now, assume without loss of generality that $w_1$ is a vertex of degree $1$ in $C$ and that $w_2$ is its unique neighbour in $G_i$. Let $x$ be an isolated vertex of $G_i$; such a vertex exists by Property (3). Let $G_{i+1} = (G_i \setminus w_1 w_2) \cup w_1 x$. Note that, by Property (4), these are the only two possibilities for $C$. Moreover, note that the number of isolated edges in $G_{i+1}$ is strictly larger than the number of such edges in $G_i$, implying that the required sequence of graphs is indeed finite. 

Fix an arbitrary index $0 \leq i < t$. In order to complete the proof of the lemma, it suffices to prove that $|S_{G_i}| \leq |S_{G_{i+1}}|$. Let $C$ denote the unique connected component of $G_i$ which was ``broken'' to obtain $G_{i+1}$, and let $V(C) = \{w_1, \ldots, w_{\ell}\}$. 
If $C$ is not a tree, then it easy to verify that $S_{G_i} \subseteq S_{G_{i+1}}$, implying the required inequality $|S_{G_i}| \leq |S_{G_{i+1}}|$. Assume then that $C$ is a tree, let $w_1$ be a vertex of degree $1$ in $C$ and let $w_2$ be its unique neighbour. Observe that if $\pi\in S_{G_i}\setminus S_{G_{i+1}}$, then $\{w_1, x\} \subseteq \{\pi(n-r+1), \ldots, \pi(n)\}$ and $\{w_2, \ldots,w_{\ell}\} \subseteq \{\pi(1), \ldots, \pi(n-r)\}$. Therefore $|S_{G_i}\setminus S_{G_{i+1}}|\leq r (r-1) (n-r) (n-r-1) \ldots (n - r - \ell + 2) (n - \ell - 1)!$. On the other hand, if $\pi \in S_n$ is such that $x \notin \{\pi(n-r+1), \ldots, \pi(n)\}$, $w_1 \in \{\pi(n-r+1), \ldots, \pi(n)\}$, and $|\{\pi(n-r+1), \ldots, \pi(n)\} \cap \{w_2, \ldots, w_{\ell}\}| = 1$, then $\pi \in S_{G_{i+1}}\setminus S_{G_i}$. Therefore, $|S_{G_{i+1}}\setminus S_{G_i}|\geq (\ell - 1) r (r-1) (n-r) (n-r-1) \ldots (n - r - \ell + 2) (n - \ell - 1)!$. It follows that in the transition from $G_i$ to $G_{i+1}$ we lost several good permutations but gained at least that many, implying the required inequality $|S_{G_i}| \leq |S_{G_{i+1}}|$. 
\end{proof}

\begin{proof} [Proof of Theorem~\ref{th::oddRegularGraph}]
Fix some integer $r \geq f(n)$. Our goal is to prove that $\Pr(\tau_{\text{lab}}(G, n) \leq (k+1) n - r) = o(1)$.
Let $\mathcal H$ denote Builder's graph after exactly $k n$ rounds, and let $H_1$ be a graph for which $\Pr(\tau_{\text{lab}}(G, n) \leq (k+1) n - r\mid {\mathcal H}=H_1)$ is maximal. 
Note that $e(G \setminus H_1)\geq e(G) - e(H_1) = n/2$  and let $H \subseteq G \setminus H_1$ be some subgraph with exactly $n/2$ edges.
Since 
\begin{align*}
\Pr(\tau_{\text{lab}}(G, n) \leq (k+1) n - r) &\leq \Pr(\tau_{\text{lab}}(G, n) \leq (k+1) n - r \mid {\mathcal H} = H_1)\\
&=\Pr(\tau_{\text{lab}}(G \setminus H_1, n) \leq n - r)\leq \Pr(\tau_{\text{lab}}(H, n) \leq n - r),
\end{align*}
it suffices to prove that $\Pr(\tau_{\text{lab}}(H, n) \leq n - r) = o(1)$. 
This clearly holds if $\Pr(\tau_{\text{lab}}(H, n) \leq n) = 0$; we may thus assume that $\Pr(\tau_{\text{lab}}(H, n) \leq n) > 0$.
In particular, there exists an orientation $D$ of $H$ such that $d^+_D(u) \leq 1$ for every $u \in V(H)$, implying that every connected component of $H$ is either a cycle or it contains a vertex of degree $1$. We conclude that $H$ satisfies all the conditions of Lemma~\ref{lem::GeneralizedPerfectMatching} and thus 
\begin{equation*}
\Pr(\tau_{\text{lab}}(G, n) \leq (k+1) n - r) \leq \Pr(\tau_{\text{lab}}(H, n) \leq n - r) = o(1).
\qedhere\end{equation*}
\end{proof}

\begin{remark}
The proof of Theorem~\ref{th::oddRegularGraph} would follow through to show that w.h.p. $\tau(G, n) \geq (k+1) n - f(n)$, if for every $(2k+1)$-regular graph $G$ on $n$ vertices and for every (labelled) graph $H_1$ with $k n$ edges on the vertex set $[n]$ there would be at most one (or only a handful) labelled graph $G_1$ on the vertex set $[n]$ such that $G_1$ is isomorphic to $G$ and $e(G_1 \setminus H_1) \leq n$ (this holds, for example, for $G = K_n$-- see Corollary~\ref{cor::Kn}).
However, even the weaker statement that for every $(2k+1)$-regular graph $G$ on $n$ vertices and for every $2k$-regular (labelled) graph $H_1$ on the vertex set $[n]$ there are only a few labelled graphs $G_1$ on the vertex set $[n]$ such that $G_1$ is isomorphic to $G$ and $G_1 \setminus H_1$ is a matching, is false.
For example, the labelled graph $2 C_4$ may be completed to a graph isomorphic to the `cube' graph $H_8$ by adding a matching in $8$ different ways.
Therefore, the labelled graph $2k C_4$ may be completed to a graph isomorphic to the graph $k H_8$ by adding a matching in $\frac{(2k)!}{k!} 4^k$ different ways. 
\end{remark}

\section{Trees} \label{sec::trees}
Recall that   
$$
n-1 \leq  \tau(T, n) \leq \tau_{\text{lab}}(T, n) \leq n
$$
holds by Corollary~\ref{cor:trees} for every tree $T$ on $n$ vertices. The remaining interesting question is to determine $\Pr(\tau(T, n) = n-1)$ and $\Pr(\tau_{\text{lab}}(T, n) = n-1)$ for every tree $T$. We make the following small step in this direction. 

\begin{proposition} \label{prop::trees}
Let $n \geq 2$ be an integer and let $T$ be a tree on the vertex set $[n]$. 
\begin{description}
\item [(a)] If $T$ is a path, then $\tau(T, n) = n-1$ and $\Pr(\tau_{\text{lab}}(T, n) = n-1) = \Theta(1/n)$;

\item [(b)] If $\tau(T, n) = n-1$, then $T$ is a path;

\item [(c)] If $T$ is a star, then 
$$
\Pr(\tau(T, n) = n-1) = \frac{1}{n-1} \left(1 + H_{n-2}\right) = (1 + o(1)) \log n/n
$$
and 
$$
\Pr(\tau_{\text{lab}}(T, n)=n-1) = \frac{1}{n} \left(1 + H_{n-1}\right) = (1 + o(1)) \log n/n.
$$
In particular, 
$$
\Pr(\tau(T, n) = n-1) - \Pr(\tau_{\text{lab}}(T, n) = n-1) = \frac{1}{n(n-1)} H_{n-2} = (1+o(1)) \log n/n^2.
$$
\end{description}
\end{proposition}

\begin{remark} \label{rem::StarvsPath}
It is interesting to note that, as can be seen from Proposition~\ref{prop::trees}, $\Pr(\tau(T, n) = n-1)$ and $\Pr(\tau_{\text{lab}}(T, n) = n-1)$ are ``very close'' when $T$ is a star but are ``very far'' when $T$ is a path. Moreover, $\Pr(\tau(K_{1,n-1}, n) = n-1) < \Pr(\tau(P_n, n) = n-1)$ but $\Pr(\tau_{\text{lab}}(K_{1,n-1}, n) = n-1) > \Pr(\tau_{\text{lab}}(P_n, n) = n-1)$.    
\end{remark}

\begin{proof} [Proof of Proposition~\ref{prop::trees}]
Starting with the first part of (a), we will describe Builder's strategy for constructing a Hamilton path in $n-1$ rounds. At any point during the process, let $F \subseteq [n]$ denote the set of vertices that were not offered until this point; that is, immediately before the $i$th round, the set $F$ consists of the $n-(i-1)$ vertices that were not offered in the first $i-1$ rounds.
For every $1 \leq i \leq n-1$, let $u_i$ denote the vertex Builder is offered in the $i$th round. In the first round, Builder claims an arbitrary edge which is incident with $u_1$. For every $i \geq 2$, Builder plays the $i$th round as follows. If $u_i$ is isolated in his graph, then Builder connects it to the endpoint of his current path which is not in $F$. Otherwise, he connects $u_i$ to an arbitrary isolated vertex.

In order to prove that $\tau(T, n) \leq n-1$, it remains to prove that Builder can indeed follow the proposed strategy and that, by doing so, he builds a Hamilton path in $n-1$ rounds. In order to do so, we will prove by induction on $i$ that, for every $1 \leq i \leq n-1$, immediately after the $i$th round, Builder's graph is a path $x_0, x_1, \ldots, x_i$ such that $\{x_0, x_1, \ldots, x_i\} \cap F = \{x_i\}$. This is clearly true for $i=1$. Assume that this is true for some $1 \leq i \leq n-2$ and let $x_0, x_1, \ldots, x_i$ be Builder's path immediately after the $i$th round. Consider the $(i+1)$th round. If $u_i$ is isolated in Builder's current graph, he connects it to $x_0$ (as $x_0 \notin F$ and $x_i \in F$ by the induction hypothesis). Otherwise, $u_i = x_i$ since by the induction hypothesis $x_i$ is the unique vertex of $F$ which is not isolated in Builder's graph. Builder then connects $x_i$ to some isolated vertex $x$. In both cases, by relabeling the names of the vertices, we see that Builder's graph immediately after the $(i+1)$th round is a path $x_0, x_1, \ldots, x_{i+1}$ such that $\{x_0, x_1, \ldots, x_{i+1}\} \cap F = \{x_{i+1}\}$. In particular, immediately after the $(n-1)$th round, Builder's graph is a Hamilton path $x_0, \ldots, x_{n-1}$. 

Next, we prove the second part of (a). Let $T$ be the path with vertex set $\{u_1, \ldots, u_n\}$ and edge set $\{u_i u_{i+1} : 1 \leq i \leq n-1\}$. For simplicity, denote $p_n := \Pr(\tau_{\text{lab}}(T,n) = n-1)$ and $a_n := n\,p_n$ for every $n \geq 1$. If the first vertex to be offered is $u_1$, then Builder must connect it to $u_2$. He then has to build the path $T \setminus u_1$; clearly the probability of doing so in $n-2$ rounds is $p_{n-1}$. Similarly, 
$$
\Pr(\tau_{\text{lab}}(T,n) = n-1 \mid u_n \textrm{ is the first vertex to be offered}) = p_{n-1}.
$$ 
If on the other hand, the first vertex to be offered is $u_i$ for some $2 \leq i \leq n-1$, then Builder has two options to choose from; he can connect $u_i$ to $u_{i-1}$ or to $u_{i+1}$. In the former case, he can then complete the path in $n-2$ additional rounds with probability $\frac{i-1}{n-1} \cdot p_{i-1}$. Indeed, if the vertex to be offered in round $n$ is in the set $\{u_1, \ldots, u_{ i-1}\}$, an event which occurs with probability $\frac{i-1}{n-1}$, then Builder can build the path $T[\{u_i, \ldots, u_n\}]$ with probability 1 (by connecting $u_j$ when it is offered to $u_{j-1}$ for every $i+1 \leq j \leq n$) and the path $T[\{u_1, \ldots, u_{i-1}\}]$ with probability $p_{i-1}$. If, on the other hand, the vertex to be offered in round $n$ is in the set $\{u_{i+1}, \ldots, u_n\}$, then building $T$ will surely require $n$ rounds. Similarly, in the latter case, Builder can complete the path in $n-2$ additional rounds with probability $\frac{n-i}{n-1} \cdot p_{n-i}$. For every $n > 1$, it then holds by the law of total probability that   
$$
p_n = \frac{2}{n} p_{n-1} + \sum_{i=2}^{n-1} \frac{1}{n} \max\left\{\frac{i-1}{n-1} p_{i-1}, \frac{n-i}{n-1} p_{n-i}\right\},
$$
or, equivalently,
\begin{equation}\label{eq:recursion}
a_n = \frac{1}{n-1}\left(2a_{n-1} + \sum_{i=2}^{n-1} \max\left\{a_{i-1}, a_{n-i}\right\}\right).
\end{equation}

Note that $p_m \geq \frac{m-1}{m} p_{m-1}$ for every $m > 1$. Indeed, if $u_m$ is offered before round $m$, then Builder can connect it to $u_{m-1}$ and complete the remainder of the path in $m-2$ rounds with probability $p_{m-1}$. It follows that $\{a_n\}_{n=1}^{\infty}$ is a non-decreasing sequence. Therefore, for every $n>1$, \eqref{eq:recursion} yields that
$$
a_n =
\begin{cases}
\frac{2}{n-1} \sum_{j=n/2}^{n-1} a_j & \quad n \text{ is even,} \\
\frac{1}{n-1} a_{(n-1)/2} + \frac{2}{n-1} \sum_{j=(n+1)/2}^{n-1} a_j & \quad n \text{ is odd.}
\end{cases}
$$

In particular, for every $k\geq 1$, it holds that
$$
2ka_{2k+1}-(2k-1)a_{2k}=\left(a_k +2 \sum_{j=k+1}^{2k} a_j\right)-2 \sum_{j=k}^{2k-1} a_j=2a_{2k}-a_k,
$$
implying that
$$
a_{2k+1} - a_{2k} = \frac{1}{2k} (a_{2k} - a_k).
$$
Moreover
\begin{align*}
(2k+1)a_{2k+2}-2\cdot 2ka_{2k+1}+(2k-1)a_{2k} &= 2 \sum_{j=k+1}^{2k+1} a_j-2\left(a_k +2 \sum_{j=k+1}^{2k} a_j\right)+2 \sum_{j=k}^{2k-1} a_j \\
&= 2a_{2k+1}-2a_{2k},
\end{align*}
implying that
$$
a_{2k+2} - a_{2k+1} = a_{2k+1} - a_{2k}.
$$

We conclude that 
$$
a_{2k+2} - a_{2k+1} = a_{2k+1} - a_{2k} = \frac{1}{2k} (a_{2k} - a_k) = \frac{1}{2k} \sum_{j=k+1}^{2k} (a_j - a_{j-1})
$$
holds for every $k \geq 1$.

Since $a_2 - a_1 = 2 - 1 = 6 - 5 = \frac{30}{5} - \frac{30}{6}$ and for every $k \geq 1$ it holds that 
\begin{align*}
\frac{1}{2k} \sum_{j=k+1}^{2k} \left(\frac{30}{j+3}-\frac{30}{j+4}\right)&=\frac{1}{2k}\left(\frac{30}{k+4}-\frac{30}{2k+4}\right) = \frac{30}{(2k+4)(2k+8)}\\
&\leq \frac{30}{2k+5}-\frac{30}{2k+6}\leq \frac{30}{2k+4}-\frac{30}{2k+5},
\end{align*}
it follows by induction that $a_n-a_{n-1}\leq\frac{30}{n+3}-\frac{30}{n+4}$ holds for every $n \geq 2$.
Therefore
$$
1 = a_1 \leq a_n \leq a_1 + \sum_{i=2}^n \left(\frac{30}{i+3} - \frac{30}{i+4} \right) =  1 + \frac{30}{5} - \frac{30}{n+4} < 7
$$ 
holds for every $n \geq 1$, and the claim follows.

Next, we prove (b). If at some point during the first $n-1$ rounds, Builder claims an edge which closes a cycle, then he cannot build $T$ in $n-1$ rounds; assume then that after $n-1$ rounds Builder's graph is a tree. We will prove by induction on $1 \leq i \leq n-1$ that, with positive probability, the following two properties hold immediately after the $i$th round:
\begin{description}
\item [(1)] Builder's graph is a path $x_0, \ldots, x_i$;

\item [(2)] 
The set of vertices that were offered in the first $i$ rounds is $\{x_0,\ldots,x_{i-1}\}$.
\end{description}
This is trivially true for $i=1$. Assume then that, for some $1 \leq i \leq n-2$, immediately after the $i$th round, Builder's graph satisfies properties (1) and (2) as above. If in round $i+1$ Builder is offered $x_i$, then immediately after this round, his graph still satisfies properties (1) and (2). This concludes the induction step as Builder is offered $x_i$ in round $i+1$ with probability $1/(n-i) > 0$. It then follows that, with probability at least $1/(n-1)! > 0$, Builder's graph after $n-1$ rounds will either contain a cycle or will be a path.   

Finally, we prove (c). Let $u_1 u_2$ be the edge Builder claims in the first round, where $u_1$ is the vertex he was offered. 

If Builder is offered the vertex $u_2$ in the second round (an event which occurs with probability $\frac{1}{n-1}$), then Builder will be able to build $T$ in $n-1$ rounds (by connecting all the vertices he is offered after the second round to $u_2$) if and only if the vertex to which $u_2$ was connected in the second round will not be offered during the next $n-3$ rounds; this event will occur with probability $\frac{1}{n-2}$.

If the vertex Builder is offered in the second round is not $u_2$, say $u_3$, then in order to have a positive probability of completing $T$ in $n-1$ rounds, Builder must connect $u_3$ to either $u_1$ or $u_2$. 

If, in the second round, Builder connects $u_3$ to $u_1$, then he will be able to build $T$ in $n-1$ rounds (by connecting all the vertices he is offered after the second round to $u_1$) if and only if $u_2$ will not be offered during the next $n-3$ rounds; this event will occur with probability $\frac{1}{n-2}$.

However, it turns out that Builder has a larger probability of building $T$ in $n-1$ rounds if he connects $u_3$ to $u_2$ in the second round. Then, in order to build $T$ in $n-1$ rounds, from this point onwards Builder must connect every vertex he is offered to $u_2$, unless he is offered $u_2$ which he should then connect to an arbitrary isolated vertex. This is possible if and only if either $u_2$ is not offered at all during the $n-3$ rounds following the second round, or $u_2$ is offered in some round $3 \leq i \leq n-1$ and then Builder connects it to some vertex that was not offered yet and this vertex happens not to be offered during the next $n-i-1$ rounds. The probability of this event is 
$$
\frac{1}{n-2} + \left(\sum_{i=3}^{n-1} \frac{1}{n-2} \cdot \frac{1}{n-i}\right)=\frac{1}{n-2} \left(1+\sum_{i=3}^{n-1} \frac{1}{n-i}\right)
$$
which is clearly larger than $\frac{1}{n-2}$.

We conclude that 
\begin{align*}
\Pr(\tau(T, n) = n-1) &=\frac{1}{n-1}\cdot \frac{1}{n-2}+\frac{n-2}{n-1}\cdot\frac{1}{n-2}\left(1 + \sum_{i=3}^{n-1} \frac{1}{n-i}\right)\\
&=\frac{1}{n-1}\left(1 +\sum_{i=2}^{n-1} \frac{1}{n-i}\right)=\frac{1}{n-1}\left(1 +H_{n-2}\right).
\end{align*}        
 
A similar (simpler) analysis shows that 
\begin{equation*}
\Pr(\tau_{\text{lab}}(T, n) = n-1) =\frac{1}{n}\left(1 +H_{n-1}\right).
\qedhere\end{equation*}         
\end{proof}

\section{Edge-connectivity} \label{sec::EdgeCon}
In this section we consider the $k$-edge-connectivity game. For every positive integer $k$, let ${\mathcal C}_k = {\mathcal C}_k(n)$ denote the family of all $k$-edge-connected $n$-vertex graphs. Since there are $k$-vertex-connected $k$-regular graphs for every $k \geq 2$, it follows by Corollary~\ref{cor::evenRegular} that $\tau(\mathcal{C}_{2k}, n) = k n$ for every positive integer $k$ and every sufficiently large $n$. Moreover, $\tau(\mathcal{C}_1, n) = n-1$. Indeed, the lower bound is trivial and the upper bound holds since $\tau(\mathcal{C}_1, n) \leq \tau(P_n, n) = n-1$, where the equality holds by Proposition~\ref{prop::trees}(a). The following theorem determines $\tau(\mathcal{C}_r, n)$ asymptotically for all other (not too small or too large) values of $r$.  

\begin{theorem} \label{th::edgeConk}
Let $n \geq 12$ and $2 \leq k \leq n/2 - 1$ be positive integers and let $f : \mathbb{N} \to \mathbb{N}$ be a function satisfying $f(n) = \omega(\sqrt{n})$. Then w.h.p.
$$
\left(k + 1 - 1/e \right) n - f(n) \leq \tau(\mathcal{C}_{2k+1}, n) \leq \left(k + 1 - 1/e \right) n + f(n) + 2k,
$$
where the upper bound holds under the additional assumption that $k = o(n)$.
\end{theorem}

\begin{remark} \label{rem::missingValues}
Note that for $r=3$ we know only that  
$$
\left(2 - 1/e \right) n - f(n) \leq \tau(\mathcal{D}_3, n) \leq \tau(\mathcal{C}_3, n) \leq \tau(\mathcal{C}_4, n) = 2n.
$$
Similarly, if $r = \Theta(n)$ is odd, then we know only that  
$$
\left((r+1)/2 - 1/e \right) n - f(n) \leq \tau(\mathcal{D}_r, n) \leq \tau(\mathcal{C}_r, n) \leq \tau(\mathcal{C}_{r+1}, n) = (r+1) n/2.
$$
\end{remark}

In the proof of Theorem~\ref{th::edgeConk} we will make use of the following construction which was introduced in~\cite{FH} (in fact, a much larger family of such graphs was considered there). For every positive integer $n$ and every integer $2 \leq t \leq n/3$ for which $m := n/t$ is an integer, let $G_t$ be the graph on $n$ vertices which is defined as follows. Its vertex set is $V_1 \cup \ldots \cup V_t$, where $V_i = \left\{u_1^i, \ldots, u_m^i \right\}$ for every $1 \leq i \leq t$. Its edge set is $E_1 \cup \ldots \cup E_t \cup E'$, where $E' = \left\{u_s^i u_s^j : 1 \leq s \leq m \textrm{ and } 1 \leq i < j \leq t \right\}$, and $E_i = \left\{u_1^i u_2^i, \ldots, u_{m-1}^i u_m^i, u_m^i u_1^i \right\}$ for every $1 \leq i \leq t$. It is easy to see that $G_t$ is $(t+1)$-regular, and it was proved in~\cite{FH} that it is also $(t+1)$-vertex-connected. Here we prove that it is ``highly'' edge-connected as well.  
\begin{claim} \label{cl::fewSmallCuts}
Let $n \geq 12$ and let $A \subseteq V(G_t)$ be a set of size $2 \leq |A| \leq n/2$. Then $e_{G_t}(A, V(G_t) \setminus A) \geq t+2$.  
\end{claim}

\begin{proof}
Fix an arbitrary set $A \subseteq V(G_t)$ of size $2 \leq |A| \leq n/2$. For every $1 \leq i \leq t$, let $X_i = A \cap V_i$ and assume without loss of generality that $|X_1| \leq \ldots \leq |X_t|$. Assume first that $X_t = V_t$, and observe that
$$
e_{G_t}(A, V(G_t) \setminus A) \geq |N_{G_t}(V_t) \setminus A| = |(V_1 \cup \ldots \cup V_t) \setminus A| \geq n/2 \geq t+2,
$$
where the second inequality holds since $|A| \leq n/2$ and the last inequality holds since $t \leq n/3$ and $n \geq 12$. Assume then that $X_t \subsetneq V_t$. Assume further that $X_{t-1} = \emptyset$ and thus $X_i = \emptyset$ for every $1 \leq i \leq t-1$, implying that $|X_t| = |A| \geq 2$. Let $x, y \in X_t$ be two arbitrary vertices for which $e_{G_t}(\{x,y\}, V_t \setminus A) \geq 2$ holds; such vertices exist since $X_t \subsetneq V_t$. Then
$$
e_{G_t}(A, V(G_t) \setminus A) \geq \sum_{i=1}^t e_{G_t}(\{x,y\}, V_i \setminus A) \geq 2 t \geq t+2,
$$
where the second inequality holds since $X_i = \emptyset$ for every $1 \leq i \leq t-1$ and since $e_{G_t}(\{x,y\}, V_t \setminus A) \geq 2$ holds by our choice of $x$ and $y$, and the last inequality holds since $t \geq 2$.

Let $r$ be the smallest integer for which $X_r \neq \emptyset$. Given the cases that were already considered, we can assume that $r \leq t-1$ and that $X_i \subsetneq V_i$ for every $1 \leq i \leq t$. Since $X_1 = \ldots = X_{r-1} = \emptyset$, it readily follows that $|N_{G_t}(x, V_i \setminus A)| = 1$ holds for every $x \in X_r \cup \ldots \cup X_t$ and every $1 \leq i \leq r-1$. Moreover, since $\emptyset \neq X_i \subsetneq V_i$ for every $r \leq i \leq t$, it readily follows that $e_{G_t}(X_i, V_i \setminus A) \geq 2$ holds for every $r \leq i \leq t$. We conclude that 
\begin{align*}
e_{G_t}(A, V(G_t) \setminus A) &= \sum_{i = r}^t \sum_{j=1}^{r-1} e_{G_t}(X_i, V_j) + \sum_{i = r}^t e_{G_t}(X_i, V_i \setminus A)  \geq (r-1) (t-r+1) + 2 (t-r+1) \\
&= (r+1)(t-r+1) \geq 2t \geq t+2,
\end{align*}
where the penultimate inequality holds since $1 \leq r \leq t-1$ and the last inequality holds since $t \geq 2$.
\end{proof}  

The following result is an immediate corollary of Claim~\ref{cl::fewSmallCuts}.
 
\begin{corollary} \label{cor::edgeConPlusOne}
Let $H$ be a graph with positive minimum degree such that $V(H) = V(G_t)$ and $E(H) \cap E(G_t) = \emptyset$. Then $G_t \cup H$ is $(t+2)$-edge-connected.
\end{corollary}

\begin{proof}
Denote $\Gamma = G_t \cup H$. We need to prove that $e_{\Gamma}(A, V(G_t) \setminus A) \geq t+2$ holds for every $A \subseteq V(G_t)$ of size $1 \leq |A| \leq n/2$; fix such a set $A$. If $|A| = 1$, then 
$$
e_{\Gamma}(A, V(G_t) \setminus A) \geq \delta(\Gamma) \geq t+2.
$$ 
If on the other hand $2 \leq |A| \leq n/2$, then $e_{\Gamma}(A, V(G_t) \setminus A) \geq e_{G_t}(A, V(G_t) \setminus A)\geq t+2$ holds by Claim~\ref{cl::fewSmallCuts}.
\end{proof}

Now that we have Corollary~\ref{cor::edgeConPlusOne} at hand, the proof of Theorem~\ref{th::edgeConk} is fairly sraightforward. 

\begin{proof} [Proof of Theorem~\ref{th::edgeConk}]
Since, clearly, $\tau(\mathcal{C}_d, n) \geq \tau(\mathcal{D}_d, n)$ for every $d$, the lower bound follows immediately from Theorem~\ref{th::minDegk}(ii). In order to prove the upper bound, we need to show that Builder has a strategy which w.h.p. enables him to build a $(2k+1)$-edge-connected graph on $n$ vertices within $\left(k + 1 - 1/e \right) n + f(n) + 2k$ rounds. Builder proceeds as follows. In the first $k n$ rounds, he builds a copy of $G_{2k-1}$; since the latter graph is $2k$-regular, this is possible by Corollary~\ref{cor::evenRegular}. 
He then builds a graph $H$ such that $E(H) \cap E(G_{2k-1}) = \emptyset$ and $\delta(H) \geq 1$. 
By Lemma \ref{lem::edgeDisjoint}, this can be done w.h.p. within $\left(1 - 1/e \right) n + f(n) + 2k$ additional rounds.
This concludes the proof of the theorem as $G_{2k-1} \cup H$ is $(2k+1)$-edge-connected by Corollary~\ref{cor::edgeConPlusOne}.    
\end{proof}

\section{Concluding remarks and open problems} \label{sec::openProbs}

In this paper we have studied a no-replacement variant of the semi-random graph process. We suggest a few related open problems for future research.

\medskip

\noindent \textbf{Labeled vs. Unlabeled.} As noted in the introduction, $\tau({\mathcal F}, n) \leq \tau_{\text{lab}}({\mathcal F}, n)$ holds for every family ${\mathcal F}$ of $n$-vertex graphs. 
We have proved that $\tau({\mathcal F}, n) = \tau_{\text{lab}}({\mathcal F}, n)$ for several such families (e.g., when ${\mathcal F}$ consists of a single regular graph of even degree and when ${\mathcal F}$ is a perfect matching). 
It would be interesting to decide whether there exists a (natural) family ${\mathcal F}$ of $n$-vertex graphs for which the gap between $\tau_{\text{lab}}({\mathcal F}, n)$ and $\tau({\mathcal F}, n)$ is substantial, say $\tau_{\text{lab}}({\mathcal F}, n) - \tau({\mathcal F}, n) = \Omega(n)$. In particular, it would be interesting to decide whether there exists an $n$-vertex regular graph of odd degree $G$ such that $\tau_{\text{lab}}(G, n) - \tau(G, n) = \Omega(n)$; recall that we have proved that $\tau_{\text{lab}}(G, n) - \tau(G, n) \leq (1/e + o(1)) n$ holds for all such graphs. While it seems quite plausible that such graph families exist, the only result we have which demonstrates that $\tau({\mathcal F}, n) < \tau_{\text{lab}}({\mathcal F}, n)$ might hold, is a tiny gap for paths. Indeed, while $\tau(P_n, n) = n-1$, the probability that $\tau_{\text{lab}}(P_n, n) = n-1$ tends to $0$ as $n$ tends to infinity.        

\medskip

\noindent \textbf{Trees.} As noted in Section~\ref{sec::trees}, the most natural and interesting question concerning trees is to determine $\Pr(\tau(T, n) = n-1)$ and $\Pr(\tau_{\text{lab}}(T, n) = n-1)$ for every tree $T$. We have proved some partial related results. In particular, we have shown that $\Pr(\tau(T, n) = n-1) = 1$ if and only if $T \cong P_n$. This implies that $P_n$ is the ``best'' tree in the sense that $\Pr(\tau(T, n) = n-1) < \Pr(\tau(P_n, n) = n-1)$ for every $n$-vertex tree $T \neq P_n$. We believe that the star $K_{1,n-1}$ is the ``worst'' tree. That is, that $\Pr(\tau(T, n) = n-1) > \Pr(\tau(K_{1,n-1}, n) = n-1)$ holds for every $n$-vertex tree $T \neq K_{1,n-1}$. 
As we saw, the situation is reversed for labeled trees, that is, $\Pr(\tau_{\text{lab}}(P_n, n) = n-1) < \Pr(\tau_{\text{lab}}(K_{1,n-1}, n) = n-1)$. It would be interesting to determine whether these are still the extremal cases, that is, whether $\Pr(\tau_{\text{lab}}(P_n, n) = n-1) < \Pr(\tau_{\text{lab}}(T, n) = n-1) < \Pr(\tau_{\text{lab}}(K_{1,n-1}, n) = n-1)$ for every $n$-vertex tree $T \notin \{P_n, K_{1,n-1}\}$.     

\subsection*{Acknowledgements} 
We are grateful to the anonymous referee for making several very helpful suggestions.

\end{document}